\documentclass{article}

\usepackage[utf8]{inputenc}

\usepackage[dvipsnames]{xcolor}

\usepackage{amsmath, amssymb, amsthm}
\usepackage{bbm}
\usepackage[giveninits=true, style=alphabetic]{biblatex}
\usepackage{blkarray}
\usepackage{cancel}
\usepackage{enumitem}
\usepackage{euscript}
\usepackage{float}
\usepackage{graphicx}
\usepackage{mathrsfs}
\usepackage{mathtools}
\usepackage{microtype}
\usepackage{ragged2e}
\usepackage{sectsty}
\usepackage{thmtools}
\usepackage{tikz}
\usepackage[Symbolsmallscale]{upgreek}

\usepackage{sinhomath}

\usepackage{hyperref}
\usetikzlibrary{cd}

\hypersetup{citecolor=red, colorlinks=true, linkcolor=blue}

\allowdisplaybreaks{}
\allsectionsfont{\raggedright}

\title{The entropic barrier is $n$-self-concordant}
\author{Sinho Chewi\footnote{The author thanks S\'ebastien Bubeck for helpful comments. The author was supported by the Department of
Defense (DoD) through the National Defense Science \& Engineering Graduate Fellowship (NDSEG) Program.} \\ \small Massachusetts Institute of Technology (MIT) \\ \small \texttt{schewi@mit.edu}}
\date{\today}

\declaretheorem[name=Lemma]{lem}
\declaretheorem[name=Proposition]{prop}

\declaretheorem[name=Theorem]{thm}

\begin{document}

\maketitle

\begin{abstract}
    For any convex body $K \subseteq \R^n$, S.\ Bubeck and R.\ Eldan introduced the entropic barrier on $K$ and showed that it is a $(1+o(1)) \, n$-self-concordant barrier.
    In this note, we observe that the optimal bound of $n$ on the self-concordance parameter holds as a consequence of the dimensional Brascamp-Lieb inequality.
\end{abstract}


\section{Introduction}

Let $K \subseteq \R^n$ be a convex body.
In~\cite{bubeckeldan2019entropic}, S.\ Bubeck and R.\ Eldan introduced the \emph{entropic barrier} $f^\star : \interior K \to \R$, defined as follows.
First, let $f : \R^n\to\R$ denote the logarithmic Laplace transform of the uniform measure on $K$,
\begin{align}\label{eq:f}
    f(\theta)
    &:= \ln \int_K \exp{\langle \theta, x \rangle} \, \D x\,.
\end{align}
Then, define $f^\star$ to be the Fenchel conjugate of $f$,
\begin{align*}
    f^\star(x)
    &:= \sup_{\theta\in\R^n}\{\langle \theta, x \rangle - f(\theta)\}\,.
\end{align*}
They proved the following result.

\begin{thm}[{\cite[Theorem 1]{bubeckeldan2019entropic}}]\label{thm:entropic_barrier}
    The function $f^\star$ is strictly convex on $\interior K$. Also, the following statements hold.
    \begin{enumerate}
        \item $f^\star$ is self-concordant, i.e.\
            \begin{align*}
                \nabla^3 f^\star(x)[h,h,h] \le 2 \, \abs{\langle h, \nabla^2 f^\star(x) \, h \rangle}^{3/2}\,, \qquad \text{for all}~x \in \interior K\,, \; h \in \R^n\,.
            \end{align*}
        \item $f^\star$ is a $\nu$-self-concordant barrier, i.e.\
            \begin{align*}
                \nabla^2 f^\star(x)
                &\succeq \frac{1}{\nu} \, \nabla f^\star(x) \, {\nabla f^\star(x)}^\T\,, \qquad \text{for all}~x \in \interior K\,,
            \end{align*}
            with $\nu = (1 + o(1)) \, n$.
    \end{enumerate}
\end{thm}

Self-concordant barriers are most well-known for their prominent role in the theory of interior-point methods for optimization~\cite{nesterov1995interiorpoint}, but they also find applications to numerous other problems such as online linear optimization with bandit feedback~\cite{abernethy2008banditlinear} (indeed, the latter was a motivating example for the introduction of the entropic barrier in~\cite{bubeckeldan2019entropic}).

A central theoretical question in the study of self-concordant barriers is: for any convex domain $K\subseteq \R^n$, does there exist a $\nu$-self-concordant barrier for $K$, and if so, what the optimal value of the parameter $\nu$?
In their seminal work~\cite{nesterov1995interiorpoint}, Y.\ Nesterov and A.\ Nemirovskii constructed for each $K$ a \emph{universal barrier} with $\nu = O(n)$.
On the other hand, explicit examples (e.g.\ the simplex and the cube) show that the best possible self-concordance parameter is $\nu = n$~\cite[Proposition 2.3.6]{nesterov1995interiorpoint}.
The situation was better understood for convex cones, on which the \emph{canonical barrier} was shown to be $n$-self-concordant independently by R.\ Hildebrand and D.\ Fox~\cite{hildebrand2014canonicalbarrier, fox2015canonical}.
Then, in~\cite{bubeckeldan2019entropic}, S.\ Bubeck and R.\ Eldan introduced the entropic barrier and showed that it is $(1 + o(1)) \, n$-self-concordant on general convex bodies, and $n$-self-concordant on convex cones; further, they showed that the universal barrier is also $n$-self-concordant on convex cones.
Subsequently, Y.\ Lee and M.\ Yue settled the question of obtaining optimal self-concordant barriers for general convex bodies by proving that the universal barrier is always $n$-self-concordant~\cite{leeyue2021universalbarrier}.

The purpose of this note is to describe the following observation.

\begin{thm}\label{thm:main}
    The entropic barrier on any convex body $K \subseteq \R^n$ is an $n$-self-concordant barrier.
\end{thm}

Besides improving the result of~\cite{bubeckeldan2019entropic}, the theorem shows that the entropic barrier provides a second example of an optimal self-concordant barrier for general convex bodies; to the best of the author's knowledge, no other optimal self-concordant barriers are known.

We will provide two distinct proofs of~\autoref{thm:main}.
First, we will observe that~\autoref{thm:main} is an immediate consequence of the following theorem, which was obtained independently in~\cite{nguyen2014dimensionalvariance, wang2014heatcapacity}; see also~\cite{fradelizimadimanwang2016infocontent}.

\begin{thm}\label{thm:varentropy}
    Let $\mu \propto \exp(-V)$ be a log-concave density on $\R^n$.
    Then,
    \begin{align*}
        \var_\mu V
        &\le n\,.
    \end{align*}
\end{thm}

In turn, as discussed in~\cite{nguyen2014dimensionalvariance, bolleygentilguillin2018brascamplieb},~\autoref{thm:varentropy} is related to certain dimensional improvements of the \emph{Brascamp-Lieb inequality}.
We state a version of this inequality which is convenient for the present discussion.

\begin{thm}[{\cite[Proposition 4.1]{bolleygentilguillin2018brascamplieb}}]\label{thm:brascamp_lieb_dim}
    Let $\mu\propto \exp(-V)$ be a log-concave density on $\R^n$, where $V$ is of class $\mc C^2$ and $\nabla^2 V \succ 0$.
    Then, for all $\mc C^1$ compactly supported $g : \R^n\to\R$, it holds that
    \begin{align*}
        \var_\mu g
        &\le \E_\mu\langle \nabla g, (\nabla^2 V){}^{-1} \, \nabla g \rangle - \frac{{\cov_\mu(g, V)}^2}{n - \var_\mu V}\,.
    \end{align*}
\end{thm}

It is straightforward to see that~\autoref{thm:brascamp_lieb_dim} implies~\autoref{thm:varentropy}.
Indeed, via a routine approximation argument, we may assume that $\mu$ satisfies the hypothesis of~\autoref{thm:brascamp_lieb_dim}.
Taking $g = V$ (which is justified via another approximation argument) and rearranging the inequality of~\autoref{thm:brascamp_lieb_dim} yields
\begin{align*}
    \var_\mu V
    &\le \frac{n \E_\mu \langle \nabla V, (\nabla^2 V){}^{-1} \, \nabla V \rangle}{n + \E_\mu \langle \nabla V, (\nabla^2 V){}^{-1} \, \nabla V \rangle}
    \le n\,.
\end{align*}

Next, in our second approach to~\autoref{thm:main}, we observe that a key step in the proof of~\autoref{thm:varentropy} given by~\cite{wang2014heatcapacity} is a tensorization principle.
It is then natural to wonder whether such a principle can be applied directly to deduce~\autoref{thm:main}.
Indeed, we have the following elementary lemma.

\begin{lem}\label{lem:tensorization}
    Suppose that for each $n\in\N^+$ and each convex body $K \subseteq \R^n$, we have a function $\phi_{n,K} : \interior K \to \R$ such that $\phi_{n,K}$ is a $\nu(n)$-self-concordant barrier for $K$.
    Also, suppose that the following consistency condition holds:
    \begin{align}\label{eq:consistency}
        \phi_{m+n, K\times K'}(x, x')
        &= \phi_{m,K}(x) + \phi_{n,K'}(x')\,,
    \end{align}
    for all $m,n\in\N^+$, all convex bodies $K \subseteq \R^m$, $K' \subseteq \R^n$, and all $x \in K$, $x' \in K'$.
    Then, $\phi_{n,K}$ is a $\inf_{k\in\N^+} \nu(kn)/k$-self-concordant barrier for $K$.
\end{lem}

We will check that the entropic barrier satisfies the consistency condition described in the previous lemma in Section~\ref{scn:tensorization}.
Combined with the second statement in~\autoref{thm:entropic_barrier}, it yields another proof of~\autoref{thm:main}.

The remainder of this note is organized as follows.
In Section~\ref{scn:from_barrier_to_bl}, we will explain the connection between~\autoref{thm:main} and~\autoref{thm:varentropy}, thereby deducing the former from the latter.
Then, so as to make this note more self-contained, in Section~\ref{scn:dim_bl} we will provide two proofs of the dimensional Brascamp-Lieb inequality (\autoref{thm:brascamp_lieb_dim}).
The first proof follows~\cite{bolleygentilguillin2018brascamplieb} and proceeds via a dimensional improvement of H\"ormander's $L^2$ method.
The second ``proof'', which is only sketched, shows how the dimensional Brascamp-Lieb inequality may be obtained from a convexity principle: the entropy functional is convex along generalized Wasserstein geodesics which arise from Bregman divergence couplings~\cite{ahnchewi2021mirrorlangevin}.
The second argument appears to be new.
Finally, in Section~\ref{scn:tensorization}, we present the tensorization argument as encapsulated in~\autoref{lem:tensorization}.

\section{From the entropic barrier to the dimensional Brascamp-Lieb inequality}\label{scn:from_barrier_to_bl}

In this section, we follow~\cite{bubeckeldan2019entropic}.
The entropic barrier has a fruitful interpretation in terms of an exponential family of probability distributions defined over the convex body $K \subseteq \R^n$.
For each $\theta\in\R^n$, we define the density $p_\theta$ on $K$ via
\begin{align}\label{eq:p_theta}
    p_\theta(x)
    &:= \frac{\exp{\langle \theta, x \rangle}}{\int_K \exp{\langle \theta, x' \rangle} \, \D x'} \, \one\{x\in K\}\,.
\end{align}
Since $f$ (defined in~\eqref{eq:f}) is essentially the logarithmic moment-generating function of $p_\theta$, then the derivatives of $f$ yield cumulants of $p_\theta$.
In particular,
\begin{align*}
    \nabla f(\theta)
    &= \E_{p_\theta} X\,, \qquad \nabla^2 f(\theta) = \cov_{p_\theta} X\,.
\end{align*}
By convex duality, the mappings $\nabla f : \R^n\to \interior K$ and $\nabla f^\star : \interior K \to\R^n$ are inverses of each other.
From the classical duality between the logarithmic moment-generating function and entropy, we can also deduce that
\begin{align*}
    f^\star(x)
    &= \eu H(p_{\nabla f^\star(x)})\,,
\end{align*}
where $\eu H$ denotes the entropy functional\footnote{Note the sign convention, which is opposite the usual one in information theory. We use this convention as it is convenient for $\eu H$ to be convex.}
\begin{align}\label{eq:entropy_functional}
    \eu H(p)
    &:= \int p\ln p\,.
\end{align}

The self-concordance parameter of $f^\star$ is the least $\nu \ge 0$ such that
\begin{align*}
    \langle \nabla f^\star(x), [\nabla^2 f^\star(x)]{}^{-1} \, \nabla f^\star(x) \rangle \le \nu\,, \qquad \text{for all}~x \in \interior K\,.
\end{align*}
Taking $x = \nabla f(\theta)$, equivalently we require
\begin{align*}
    \langle \theta, \nabla^2 f(\theta) \, \theta \rangle \le \nu\,, \qquad \text{for all}~\theta \in \R^n\,,
\end{align*}
which has the probabilistic interpretation
\begin{align}\label{eq:prob_interp}
    \var_{p_\theta}{\langle \theta, X \rangle} \le \nu\,, \qquad\text{for all}~\theta \in \R^n\,.
\end{align}

From the definition~\eqref{eq:p_theta}, we see that the density $p_\theta \propto \exp(-V)$ is log-concave, where $V(x) = \langle \theta, x \rangle$ for $x\in \interior K$.
By applying~\autoref{thm:varentropy} to $p_\theta$, we immediately deduce that~\eqref{eq:prob_interp} holds with $\nu = n$.

\section{Proof of the dimensional Brascamp-Lieb inequality}\label{scn:dim_bl}

Next, we wish to give some proofs of the dimensional Brascamp-Lieb inequality (\autoref{thm:brascamp_lieb_dim}).
Classically, the Brascamp-Lieb inequality reads as follows.

\begin{thm}[{\cite{brascamplieb1976}}]\label{thm:brascamp_lieb}
    Let $\mu \propto \exp(-V)$ be a density on $\R^n$, where $V$ is a convex function of class $\mc C^2$.
    Then, for every locally Lipschitz $g : \R^n\to\R$,
    \begin{align}\label{eq:brascamp_lieb}
        \var_\mu g
        &\le \E_\mu \langle \nabla g, (\nabla^2 V){}^{-1} \, \nabla g \rangle\,.
    \end{align}
\end{thm}

The Brascamp-Lieb inequality is a Poincar\'e inequality for the measure $\mu$ corresponding to the Newton-Langevin diffusion~\cite{chewietal2020mirrorlangevin}.
When $V$ is strongly convex, $\nabla^2 V \succeq \alpha I_n$, it recovers the usual Poincar\'e inequality
\begin{align*}
    \var_\mu g \le \frac{1}{\alpha} \E_\mu[\norm{\nabla g}^2]\,.
\end{align*}
See~\cite{bobkovledoux2000brunnmintobrascamplieblsi, bakrygentilledoux2014, cordero2017transport} for various proofs of~\autoref{thm:brascamp_lieb}.

Since the inequality~\eqref{eq:brascamp_lieb} makes no explicit reference to the dimension, it actually holds in infinite-dimensional space.
In contrast,~\autoref{thm:brascamp_lieb_dim} asserts that~\eqref{eq:brascamp_lieb} can be improved by subtracting an additional non-negative term from the right-hand side in any finite dimension.
This is referred to as a \emph{dimensional improvement} of the Brascamp-Lieb inequality.

\subsection{Proof by H\"ormander's \texorpdfstring{$L^2$}{L\^{}2} method}

We now present the proof of~\autoref{thm:brascamp_lieb_dim} given in~\cite{bolleygentilguillin2018brascamplieb}.
The starting point for H\"ormander's $L^2$ method is to first dualize the Poincar\'e inequality.

\begin{prop}[{\cite[Lemma 1]{barthecorderoerausquin2013invariances}}]\label{prop:hormander}
    Let $\mu \propto \exp(-V)$ be a probability density on $\R^n$, where $V$ is of class $\mc C^1$.
    Define the corresponding generator $\ms L$ on smooth functions $g : \R^n\to\R$ via
    \begin{align*}
        \ms L g
        &:= -\Delta g + \langle \nabla V, \nabla g \rangle\,.
    \end{align*}
    Suppose $A : \R^n\to \on{PD}(n)$ is a matrix-valued function mapping into the space of symmetric positive definite matrices such that for all smooth $u : \R^n\to\R$,
    \begin{align}\label{eq:hormander}
        \E_\mu[{(\ms L u)}^2]
        &\ge \E_\mu\langle \nabla u, A \,\nabla u \rangle\,.
    \end{align}
    Then, for all $g \in L^2(\mu)$, it holds that
    \begin{align*}
        \var_\mu g
        &\le \E_\mu\langle \nabla g, A^{-1} \, \nabla g \rangle\,.
    \end{align*}
\end{prop}
\begin{proof}
    We may assume $\E_\mu g = 0$.
    This condition is certainly necessary for the equation $\ms L u = g$ to be solvable; in order to streamline the proof, we will assume that a solution $u$ exists.
    (This assumption can be avoided by invoking~\cite{corderoerausquinfradelizimaurey2014bconj} and using a density argument; see~\cite{barthecorderoerausquin2013invariances} for details.)

    Using the integration by parts formula for the generator,
    \begin{align*}
        \E_\mu[g \, \ms L u] = \E_\mu \langle \nabla g, \nabla u \rangle\,,
    \end{align*}
    we obtain
    \begin{align*}
        \var_\mu g
        &= \E_\mu[g^2]
        = 2\E_\mu[g\, \ms L u] - \E_\mu[{(\ms L u)}^2]
        \le 2\E_\mu \langle \nabla g, \nabla u \rangle - \E_\mu\langle \nabla u, A \, \nabla u \rangle\,.
    \end{align*}
    Next, since $2 \, \langle x, y \rangle \le \langle x, A \, x \rangle + \langle y, A^{-1} \, y \rangle$ for all $x,y\in\R^n$, it implies
    \begin{align*}
        \var_\mu g
        &\le \E_\mu\langle \nabla g, A^{-1} \, \nabla g \rangle\,. \qedhere
    \end{align*}
\end{proof}

The key idea now is that the condition~\eqref{eq:hormander} can be verified with the help of the \emph{curvature} of the potential $V$.
Indeed, assume now that $V$ is of class $\mc C^2$ and that $\nabla^2 V \succ 0$.
By direct calculation, one verifies the commutation relation
\begin{align}\label{eq:commutation}
    \nabla \ms L u
    &= (\ms L + \nabla^2 V) \, \nabla u\,.
\end{align}
Hence,
\begin{align}\label{eq:hormander_key}
    \begin{aligned}
        \E_\mu[{(\ms L u)}^2]
        &= \E_\mu \langle \nabla u, \nabla \ms L u \rangle
        = \E_\mu\langle \nabla u, (\ms L + \nabla^2 V) \, \nabla u \rangle \\
        &= \E_\mu\langle \nabla u, \nabla^2 V \, \nabla u\rangle + \E_\mu[\norm{\nabla^2 u}_{\rm HS}^2]\,,
    \end{aligned}
\end{align}
where the last equality follows from the integration by parts formula for the generator applied to each coordinate separately: $\E_\mu\langle \nabla u, \ms L \nabla u \rangle = \E_\mu[\norm{\nabla^2 u}_{\rm HS}^2]$.
Since the second term is non-negative,~\autoref{prop:hormander} now implies the Brascamp-Lieb inequality (\autoref{thm:brascamp_lieb}).

In order to obtain the dimensional improvement of the Brascamp-Lieb inequality (\autoref{thm:brascamp_lieb_dim}), we will imitate the proof of~\autoref{prop:hormander}, only now we will use the additional term $\E_\mu[\norm{\nabla^2 u}_{\rm HS}^2]$ in the above identity.

\begin{proof}[Proof of~\autoref{thm:brascamp_lieb_dim}]
    As before, let $\E_\mu g = 0$.
    However, we introduce an additional trick and consider $u$ not necessarily satisfying $\ms L u = g$; this will help to optimize the bound at the end of the argument.
    Following the computations in~\autoref{prop:hormander} and using the key identity~\eqref{eq:hormander_key}, we obtain
    \begin{align*}
        \var_\mu g
        &= \E_\mu[g^2]
        = \E_\mu[{(g - \ms L u)}^2] + 2\E_\mu[g \, \ms L u] - \E_\mu[{(\ms L u)}^2] \\
        &= \E_\mu[{(g - \ms L u)}^2] + 2\E_\mu\langle \nabla g, \nabla u\rangle - \E_\mu \langle \nabla u, \nabla^2 V \, \nabla u \rangle - \E_\mu[\norm{\nabla u}_{\rm HS}^2] \\
        &\le \E_\mu[{(g - \ms L u)}^2] + \E_\mu\langle \nabla g, (\nabla^2 V){}^{-1} \, \nabla g\rangle - \E_\mu[\norm{\nabla u}_{\rm HS}^2]\,.
    \end{align*}
    For the second term, we use the inequality
    \begin{align*}
        \E_\mu[\norm{\nabla u}_{\rm HS}^2]
        \ge \frac{1}{n} \, {(\E_\mu \Delta u)}^2\,.
    \end{align*}
    From integration by parts,
    \begin{align*}
        \E_\mu \Delta u
        &= \E_\mu\langle \nabla V, \nabla u \rangle
        = \E_\mu[V \, \ms L u]
        = \cov_\mu(g, V) + \E_\mu[V \, (\ms L u - g)]\,.
    \end{align*}
    We now choose $\ms L u = g + a \, (V - \E_\mu V)$ for some $a \ge 0$ to be chosen later.
    For brevity of notation, write $\mb C := \cov_\mu(g, V)$ and $\mb V := \var_\mu V$.
    Then,
    \begin{align*}
        &\var_\mu g - \E_\mu\langle \nabla g, (\nabla^2 V){}^{-1} \,\nabla g \rangle
        \le a^2 \mb V - \frac{1}{n} \, {(\mb C + a\mb V)}^2 \\
        &\qquad = -\frac{\mb V \, (n-\mb V)}{n} \, {\Bigl( a - \frac{\mb C}{n-\mb V} \Bigr)}^2 - \frac{\mb C^2 \mb V}{n \, (n-\mb V)} - \frac{\mb C^2}{n}\,.
    \end{align*}
    Observe that this inequality entails $\mb V \le n$, or else we could send $a\to\infty$ and arrive at a contradiction.
    Optimizing over $a$, we obtain
    \begin{align*}
        \var_\mu g
        &\le \E_\mu\langle \nabla g, (\nabla^2 V){}^{-1} \, \nabla g \rangle - \frac{\mb C^2}{n-\mb V}\,. \qedhere
    \end{align*}
\end{proof}

\subsection{Proof by convexity of the entropy along Bregman divergence couplings}

It is well-known that Poincar\'e inequalities are obtained from linearizing transportation inequalities.
In~\cite{cordero2017transport}, D.\ Cordero-Erausquin obtained the Brascamp-Lieb inequality (\autoref{thm:brascamp_lieb}) by linearizing the following inequality:
\begin{align}\label{eq:bregman_transport_ineq}
    \eu D_V(\rho \mmid \mu)
    &\le \msf{KL}(\rho \mmid \mu)\,, \qquad\text{for all}~\rho \in \mc P(\R^n)\,.
\end{align}
Here, $\mu \propto \exp(-V)$ on $\R^n$; $\mc P(\R^n)$ denotes the space of probability measures on $\R^n$; $\msf{KL}(\cdot \mmid \cdot)$ is the Kullback-Leibler (KL) divergence; and $\eu D_V(\cdot \mmid \cdot)$ is the Bregman divergence coupling cost, defined as
\begin{align*}
    \eu D_V(\rho \mmid \mu)
    &= \inf_{\gamma \in \msf{couplings}(\rho, \mu)} \int D_V(x,y) \, \D \gamma(x,y)\,,
\end{align*}
with
\begin{align*}
    D_V(x,y)
    &:= V(x) - V(y) - \langle \nabla V(y), x-y \rangle\,.
\end{align*}

On the other hand, together with K.\ Ahn in~\cite{ahnchewi2021mirrorlangevin}, the author obtained the transportation inequality~\eqref{eq:bregman_transport_ineq} as a consequence of a convexity principle in optimal transport.
It is therefore natural to ask whether the dimensional Brascamp-Lieb inequality (\autoref{thm:brascamp_lieb_dim}) can be obtained directly from (a strengthening of) this principle.
This is indeed the case, and it is the goal of the present section to describe this argument.

Making the argument fully rigorous, however, would entail substantial technical complications which would detract from the focus of this note.
In any case, a complete proof of the dimensional Brascamp-Lieb inequality is already present in~\cite{bolleygentilguillin2018brascamplieb}. Hence, we will work on a purely formal level and assume that everything is smooth, bounded, etc.
Also, the computations are rather similar to the proof of~\autoref{thm:brascamp_lieb_dim} given in the previous section.
Nevertheless, the argument seems interesting enough to warrant presenting it here.

The main difference with the preceding proof is that the Bochner formula (implicit in the commutation relation~\eqref{eq:commutation}) is replaced by the convexity principle.

\begin{proof}[Proof sketch of~\autoref{thm:brascamp_lieb_dim}]
    Throughout the proof, let $\varepsilon > 0$ be small.
    Let $h$ be bounded and satisfy $\E_\mu h = 0$, so that $\mu_\varepsilon := (1+\varepsilon h) \, \mu$ defines a valid probability density on $\R^n$.
    Our aim is to first strengthen the transportation inequality~\eqref{eq:bregman_transport_ineq}, at least infinitesimally, and then to linearize it.

    Let $(X_\varepsilon, X)$ be an optimal coupling for the Bregman divergence coupling cost $\eu D_V(\mu_\varepsilon \mmid \mu)$.
    In~\cite{ahnchewi2021mirrorlangevin}, the following facts were proven:
    \begin{enumerate}
        \item There is a function $u_\varepsilon : \R^n\to\R$ such that $\nabla V(X) = \nabla V(X_\varepsilon) - \nabla u_\varepsilon(X_\varepsilon)$, and $V - u_\varepsilon$ is convex.
        \item The entropy functional (defined in~\eqref{eq:entropy_functional}) is convex in the sense that
            \begin{align}\label{eq:cvxty_entropy}
                \eu H(\mu_\varepsilon)
                &\ge \eu H(\mu) + \E\langle [\nabla_{W_2} \eu H(\mu)](X), X_\varepsilon-X \rangle\,.
            \end{align}
            Here, $\nabla_{W_2} \eu H(\mu) = \nabla \ln \mu$ is the Wasserstein gradient of the entropy functional, c.f.~\cite{ambrosio2008gradient, villani2009ot, santambrogio2015ot}.
    \end{enumerate}
    Write $T_\varepsilon(x) := {(\nabla V - \nabla u_\varepsilon)}^{-1}(\nabla V(x))$.
    Since ${(T_\varepsilon)}_\# \mu = \mu_\varepsilon$, the change of variables formula implies
    \begin{align}\label{eq:change_of_var}
        \frac{\mu(x)}{\mu_\varepsilon(T_\varepsilon(x))}
        &= \frac{\mu(x)}{\mu(T_\varepsilon(x)) \, (1+\varepsilon h(T_\varepsilon(x)))}
        = \det \nabla T_\varepsilon(x)\,.
    \end{align}
    To linearize this equation, write $u_\varepsilon = \varepsilon u + o(\varepsilon)$ and $T_\varepsilon(x) = x + \varepsilon T(x) + o(\varepsilon)$.
    Then, the definition of $T_\varepsilon$ yields
    \begin{align*}
        \nabla V(x)
        &= (\nabla V - \nabla u_\varepsilon)\bigl(x+\varepsilon T(x) + o(\varepsilon)\bigr) \\
        &= \nabla V(x) + \varepsilon\, \nabla^2 V(x) \, T(x) - \varepsilon \, \nabla u(x) + o(\varepsilon)
    \end{align*}
    which implies
    \begin{align*}
        T_\varepsilon(x)
        &= x + \varepsilon \, [\nabla^2 V(x)]{}^{-1} \, \nabla u(x) + o(\varepsilon)\,.
    \end{align*}
    Taking logarithms and expanding to first order in $\varepsilon$,
    \begin{align*}
        &\ln \mu(x) - \ln \mu(T_\varepsilon(x)) - \ln(1+\varepsilon h(T_\varepsilon(x))) \\
        &\qquad = -\varepsilon \, \langle \nabla \ln \mu(x), [\nabla^2 V(x)]{}^{-1} \, \nabla u(x) \rangle - \varepsilon h(x) + o(\varepsilon) \\
        &\qquad = \varepsilon \, \langle \nabla V(x), [\nabla^2 V(x)]{}^{-1} \, \nabla u(x) \rangle - \varepsilon h(x) + o(\varepsilon)
    \end{align*}
    and
    \begin{align*}
        \ln \det \nabla T_\varepsilon(x)
        &= \ln \det \nabla\bigl({\id} + \varepsilon \, [\nabla^2 V]{}^{-1} \, \nabla u + o(\varepsilon)\bigr)(x) \\
        &= \ln \det\bigl(I_n + \varepsilon \, \nabla([\nabla^2 V]{}^{-1} \, \nabla u)(x) + o(\varepsilon)\bigr) \\
        &= \varepsilon \divergence([\nabla^2 V]{}^{-1} \, \nabla u)(x) + o(\varepsilon)\,.
    \end{align*}
    To interpret this, we introduce a new generator, denoted $\hat{\ms L}$ to avoid confusion with the previous section, defined by
    \begin{align*}
        \hat{\ms L} u
        &:= \divergence([\nabla^2 V]{}^{-1} \, \nabla u) - \langle \nabla V, [\nabla^2 V]{}^{-1} \, \nabla u\rangle\,.
    \end{align*}
    This new generator satisfies the integration by parts formula
    \begin{align*}
        \E_\mu[u \, \hat{\ms L} v]
        &= \E_\mu\langle \nabla u, [\nabla^2 V]{}^{-1} \, \nabla v \rangle\,.
    \end{align*}
    In this notation, the preceding computations yield
    \begin{align*}
        \hat{\ms L} u
        &= -h + o(1)\,.
    \end{align*}

    Next, to strengthen~\eqref{eq:cvxty_entropy}, we repeat the proof.
    From~\eqref{eq:change_of_var},
    \begin{align*}
        \eu H(\mu_\varepsilon)
        &= \int \mu_\varepsilon \ln \mu_\varepsilon
        = \int \mu \ln(\mu_\varepsilon \circ T_\varepsilon)
        = \int \mu \ln \frac{\mu}{\det \nabla T_\varepsilon} \\
        &= \eu H(\mu) - \int \mu \ln \det \nabla T_\varepsilon\,.
    \end{align*}
    From the second-order expansion of $-\ln \det$ around $I_n$,
    \begin{align*}
        &-\int \mu \ln \det \nabla T_\varepsilon \\
        &\qquad \ge -\int \mu \ln \det I_n - \int \mu \, \langle I_n, \nabla T_\varepsilon - I_n \rangle
        + \frac{1}{2} \int \mu \, \norm{\nabla T_\varepsilon - I_n}_{\rm HS}^2 + o(\varepsilon^2) \\
        &\qquad \ge -\int \mu \tr(\nabla T_\varepsilon - I_n) + \frac{1}{2n} \, {\Bigl(\int \mu \tr(\nabla T_\varepsilon - I_d)\Bigr)}^2 + o(\varepsilon^2) \\
        &\qquad = -\int \mu \divergence(T_\varepsilon - {\id}) + \frac{1}{2n} \, {\Bigl(\int \mu \divergence(T_\varepsilon - {\id})\Bigr)}^2 + o(\varepsilon^2) \\
        &\qquad = \int \mu \, \langle \nabla \ln \mu, T_\varepsilon - {\id} \rangle + \frac{1}{2n} \, {\Bigl(\int \mu \, \langle \nabla \ln \mu, T_\varepsilon - {\id} \rangle\Bigr)}^2 + o(\varepsilon^2)\,.
    \end{align*}
    Recalling that $\nabla_{W_2} \eu H(\mu) = \nabla \ln \mu$, we have established
    \begin{align*}
        &\eu H(\mu_\varepsilon) - \eu H(\mu) - \E\langle [\nabla_{W_2} \eu H(\mu)](X), X_\varepsilon-X \rangle \\
        &\qquad\qquad\qquad\qquad\qquad \ge \frac{1}{2n} \, {\Bigl(\int \mu \, \langle \nabla V, T_\varepsilon - {\id} \rangle\Bigr)}^2 + o(\varepsilon^2) \\
        &\qquad\qquad\qquad\qquad\qquad = \frac{\varepsilon^2}{2n} \, {\Bigl(\int \mu \, \langle \nabla V, [\nabla^2 V]{}^{-1} \, \nabla u \rangle\Bigr)}^2 + o(\varepsilon^2) \\
        &\qquad\qquad\qquad\qquad\qquad = \frac{\varepsilon^2}{2n} \, \{\E_\mu[V \, \hat{\ms L} u]\}{}^2 + o(\varepsilon^2)\,.
    \end{align*}

    The next step is to write down the strengthened transportation inequality.
    Indeed, if we add a suitable additive constant to $V$ so that $\mu = \exp(-V)$, then
    \begin{align*}
        \msf{KL}(\mu_\varepsilon \mmid \mu)
        &= \E_{\mu_\varepsilon} V + \eu H(\mu_\varepsilon) \\
        &\ge \underbrace{\E V(X) + \eu H(\mu)}_{=\msf{KL}(\mu \mmid \mu) = 0} + \underbrace{\E\langle [\nabla V + \nabla_{W_2} \eu H(\mu)](X), X_\varepsilon - X \rangle}_{= [\nabla_{W_2} \msf{KL}(\cdot \mmid \mu)](\mu) = 0} \\
        &\qquad{} + \underbrace{\E[V(X_\varepsilon) - V(X) - \langle \nabla V(X), X_\varepsilon - X \rangle]}_{=\eu D_V(\mu_\varepsilon \mmid \mu)} \\
        &\qquad{}+ \frac{\varepsilon^2}{2n} \, {\{\E_\mu[hV]\}}^2 + o(\varepsilon^2) \\
        &\ge \eu D_V(\mu_\varepsilon \mmid \mu) + \frac{\varepsilon^2}{2n} \, {\{\E_\mu[hV]\}}^2 + o(\varepsilon^2)\,.
    \end{align*}

    Finally, it remains to linearize the transportation inequality.
    On one hand, it is classical that
    \begin{align*}
        \msf{KL}(\mu_\varepsilon \mmid\mu)
        &= \frac{\varepsilon^2}{2} \E_\mu[h^2] + o(\varepsilon^2)\,.
    \end{align*}
    On the other hand, we can guess that
    \begin{align*}
        \eu D_V(\mu_\varepsilon \mmid \mu)
        &= \frac{1}{2} \E\langle X_\varepsilon - X, \nabla^2 V(X) \, (X_\varepsilon - X)\rangle + o(\varepsilon^2) \\
        &= \frac{\varepsilon^2}{2} \E_\mu\langle \nabla u, (\nabla^2 V){}^{-1} \, \nabla u \rangle + o(\varepsilon^2) \\
        &\ge \frac{\varepsilon^2}{2} \, \frac{{\{\E_\mu\langle \nabla g, (\nabla^2 V){}^{-1} \, \nabla u \rangle\}}^2}{\E_\mu\langle \nabla g, (\nabla^2 V){}^{-1} \, \nabla g\rangle} + o(\varepsilon^2) \\
        &= \frac{\varepsilon^2}{2} \, \frac{{\{\E_\mu[g \, \hat{\ms L} u]\}}^2}{\E_\mu\langle \nabla g, (\nabla^2 V){}^{-1} \, \nabla g\rangle} + o(\varepsilon^2) \\
        &= \frac{\varepsilon^2}{2} \, \frac{{\{\E_\mu[gh]\}}^2}{\E_\mu\langle \nabla g, (\nabla^2 V){}^{-1} \, \nabla g\rangle} + o(\varepsilon^2)\,.
    \end{align*}
    A rigorous proof of this inequality is given as~\cite[Lemma 3.1]{cordero2017transport}.

    Thus, we obtain
    \begin{align*}
        \frac{1}{2} \, \frac{{\{\E_\mu[gh]\}}^2}{\E_\mu\langle \nabla g, (\nabla^2 V){}^{-1} \, \nabla g\rangle} + \frac{1}{2n} \, {\{\E_\mu[hV]\}}^2
        &\le \frac{1}{2} \E_\mu[h^2] + o(1)\,.
    \end{align*}
    Now we let $\varepsilon \searrow 0$ and choose $h = g + a \, (V - \E_\mu V)$ for some $a\in\R$.
    Writing $\mb C := \cov_\mu(g, V)$ and $\mb V := \var_\mu V$, it yields
    \begin{align*}
        \frac{{(\var_\mu g + a\mb C)}^2}{\E_\mu\langle \nabla g, (\nabla^2 V){}^{-1} \, \nabla g\rangle} + \frac{1}{n} \, {(\mb C + a\mb V)}^2
        &\le \var_\mu g + 2a\mb C + a^2 \mb V\,.
    \end{align*}
    Actually, choosing $a$ to optimize this inequality and simplifying the resulting expression may be cumbersome, so with our foresight from the earlier proof of~\autoref{thm:brascamp_lieb_dim}, we now take $a = \mb C/(n-\mb V)$.
    After some algebra,
    \begin{align*}
        \frac{{(\var_\mu g + \mb C^2/(n-\mb V))}^2}{\E_\mu\langle \nabla g, (\nabla^2 V){}^{-1} \, \nabla g\rangle}
        &\le \var_\mu g + \frac{\mb C^2}{n-\mb V}\,,
    \end{align*}
    which of course yields
    \begin{align*}
        \var_\mu g
        &\le \E_\mu\langle \nabla g, (\nabla^2 V){}^{-1} \, \nabla g\rangle - \frac{\mb C^2}{n-\mb V}\,. \qedhere
    \end{align*}
\end{proof}

\section{A tensorization trick}\label{scn:tensorization}

We begin by verifying that the entropic barrier has the consistency property~\eqref{eq:consistency}.
Let $f_K$ denote the function~\eqref{eq:f}, where we now explicitly denote the dependence on the convex body $K$.
Also, let $f_K^\star$ denote the corresponding entropic barrier.
Then, we see that
\begin{align*}
    f_{K\times K'}(\theta,\theta')
    &= \ln\int_{K\times K'} \exp(\langle \theta, x \rangle + \langle \theta', x' \rangle) \, \D x \, \D x' \\
    &= \ln \int_K \exp{\langle \theta, x \rangle} \, \D x + \ln \int_{K'} \exp{\langle \theta', x' \rangle} \, \D x'
    = f_K(\theta) + f_{K'}(\theta')\,.
\end{align*}
Hence,
\begin{align*}
    f_{K\times K'}^\star(x,x')
    &= \sup_{\theta,\theta' \in \R^n}\{\langle \theta, x \rangle + \langle \theta', x' \rangle - f_K(\theta) - f_{K'}(\theta')\}
    = f_K^\star(x) + f_{K'}^\star(x')\,.
\end{align*}

Finally, we check that the tensorization property automatically improves the bound on the self-concordance parameter of $f_K^\star$ obtained in~\cite{bubeckeldan2019entropic}.

\begin{proof}[Proof of~\autoref{lem:tensorization}]
    Let $\bs x := (x_1,\dotsc,x_k) \in {(\R^n)}^k$.
    By assumption, the self-concordant barrier $\phi_{kn, K^k}$ on $K^k$ satisfies $\phi_{kn,K^k}(\bs x) = \sum_{j=1}^k \phi_{n,K}(x_i)$.
    Also, we are given that
    \begin{align}\label{eq:sc_high_dim}
        \nabla^2 \phi_{kn,K^k}(\bs x)
        &\succeq \frac{1}{\nu(kn)} \, \nabla \phi_{kn, K^k}(\bs x) \, {\nabla \phi_{kn, K^k}(\bs x)}^\T\,.
    \end{align}
    Via elementary calculations,
    \begin{align*}
        \nabla \phi_{kn,K^k}(\bs x)
        &= \bigl(\nabla \phi_{n,K}(x_1),\dotsc,\nabla \phi_{n,K}(x_k)\bigr)
    \end{align*}
    and
    \begin{align*}
        \nabla^2 \phi_{kn,K^k}(\bs x)
        &= \begin{bmatrix} \nabla^2 \phi_{n,K}(x_1) & & \\ & \ddots & \\ && \nabla^2 \phi_{n,K}(x_k) \end{bmatrix}\,.
    \end{align*}
    Let $v \in \R^n$ and let $\bs v := (v,\dotsc,v) \in {(\R^n)}^k$.
    Also, take $x_1=\cdots = x_k = x$.
    By~\eqref{eq:sc_high_dim}, we know that
    \begin{align*}
        k \, \langle v, \nabla^2 \phi_{n,K}(x) \, v \rangle
        &= \langle \bs v, \nabla^2 \phi_{nk,K^k}(\bs x) \, \bs v \rangle
        \ge \frac{1}{\nu(kn)} \, \langle \bs v, \nabla \phi_{kn, K^k}(\bs x) \rangle^2 \\
        &= \frac{k^2}{\nu(kn)} \, \langle v, \nabla \phi_{n,K}(x) \rangle^2
    \end{align*}
    which proves
    \begin{align*}
        \nabla^2 \phi_{n,K}(x)
        &\succeq \frac{k}{\nu(kn)} \, \nabla \phi_{n,K}(x) \, {\nabla \phi_{n,K}(x)}^\T
    \end{align*}
    and gives the claim.
\end{proof}

\begin{proof}[Proof of~\autoref{thm:main}]
    According to~\autoref{thm:entropic_barrier}, we know that the entropic barrier in $n$ dimensions is $(1 + \varepsilon_n) \, n$-self-concordant, with $\varepsilon_n \to 0$ as $n\to\infty$.
    By~\autoref{lem:tensorization}, it is actually $(1+\varepsilon_{kn}) \, n$-self-concordant, for any $k\in\N^+$.
    Let $k\to\infty$ to deduce that it is in fact $n$-self-concordant.
\end{proof}

\RaggedRight{}
\printbibliography{}

\end{document}